\documentclass[12pt,epsfig]{amsart}
\usepackage[dvips]{graphics}
\usepackage{epsfig}
\usepackage{latexsym}
\usepackage{amsmath,amscd}

\usepackage{epsfig}

\usepackage{amssymb}

\usepackage{latexsym}

\include{plaquettemacro}

\usepackage{hyperref}

\newtheorem{theorem}{Theorem}[section]
\newtheorem{prop}{Proposition}[section]

\newtheorem{cor}[theorem]{Corollary}

\numberwithin{equation}{section}

\newcommand{{\tlg}}{\tilde\gamma }

\newcommand{{\tlG}}{\tilde\Gamma }
\newcommand {{\tlJ}}{\tilde J }

\newcommand{\mpM}{{\mathcal P}M}

\def\slash#1{\, /\kern-0.6em{#1}}

\begin{document}

\title{Geodesics on path spaces and double category}

\author{Saikat Chatterjee}
\address{{\rm  School of Mathematics\\
The Indian Institute of Science Education and Research\\ 
Thiruvananthapuram\\
CET campus, Kerala-695016\\
India}}
\email{saikat.chat01@gmail.com}

\date{25th August, 2015}

\title{Geodesics on path spaces and a double category}




\begin{abstract}
Let $M$ be a Riemannian manifold  and $\mpM$ be the space of all smooth paths on $M$. We describe geodesics 
on path space $\mpM$. Normal neighbourhoods on $\mpM$ have been discussed. We identify paths on $M$ under
``{\em back-track}'' equivalence. Under this identification we show that if
$M$ is complete, then geodesics on the path space yield a double category. This double category has a natural
interpretation in terms of the world sheets generated by freely moving (without any external force) strings. 
\end{abstract}
\maketitle
\section{Introduction}

Let $M$ be a Riemannian manifold. We define the path space $\mpM$ over $M$ to be $C^{\infty}([0, 1], M).$
 The manifold structure on path space has been explored
in~\cite{Michor}[Chapter 10] \cite{Eliasson}[ Theorem 10.4]. We do not address the issue of 
manifold structure on $\mpM$. This paper mainly concerns with the geodesics on the 
path space over a Riemannian manifold and  a double category defined by the geodesics on $\mpM$.

Section~\ref{s:vectcon} is expository, we mostly review  known results to set up our terminologies and notations. In fact, 
a discussion with a more general framework  is available in~ \cite{freed}, \cite{spera}. We introduce an $L^2$ metric~\cite{wurzbacher, biswas-chat, munoz}
given by $${\widetilde g}(X, Y)(\gamma):=\int_0^{1}g_{\gamma}(X(t), Y(t))dt,$$
 on the path space, where $g$ is a metric on manifold $M,$ $\gamma\in \mpM$ and $X, Y\in T_{\gamma}\mpM$ (naturally identified with vector fields along
$\gamma$). A covariant connection on $\mpM$ is defined by point-wise evaluation of a covariant connection on $M$. It follows that 
on the path space, a geodesic is uniquely determined by specifying a path $\gamma\in \mpM$ and a vector field along
$\gamma$.  Proposition~\ref{prop:completeness} [Corollary A.4, \cite{freed}] shows if ambient space $M$ is complete with respect 
to a Riemannian connection, then the path space $\mpM$ is also complete with respect to the corresponding  induced connection on the path space. We also 
discussed the exponential map on the path space.

In section~\ref{s:metric} we introduce a distance function on $\mpM$. A consequence of the construction in section~\ref{s:vectcon} is that the 
exponential map $\rm Exp$ on $\mpM$ is given by 
$$\left( {\rm Exp}_{\gamma}(X)\right)(t)={\rm exp}_{\gamma(t)}X(t),\qquad \forall t\in [0,1],$$
where $\rm exp$ is the exponential map on $M$ and other notations have obvious meaning. Thus, the normal neighbourhood on $\mpM$  
is as described in Proposition~\ref{prop:existencenormalne}. In Theorem~\ref{theo:normalneighbour} we prove that any $\gamma_1, \gamma_2\in {\mathcal U_{\gamma_0}}$ can be joined  by a unique minimizing geodesic and ${\mathcal U_{\gamma_0}}$ is convex (with respect to the distance function on $\mpM$), where ${\mathcal U_{\gamma_0}}$ is a normal neighbourhood around $\gamma_0\in \mpM$.  

A prominent direction of enquiry in the area of parallel transport on path spaces has been in terms of higher categories; works in this direction include \cite{Awody, BaezSchr, BaezSchr2, AbbasWag, saikat3, saikat4} and many others. For instance, in~\cite{saikat3} a connection has been defined on the principal bundle over the path space, then it has been shown that horizontal lifting of  paths on the path space result in a double category. The intuitive  reason behind appearence of higher categories in this context is as follows. Since a path on the path space $\Gamma:[a,b]\rightarrow \mpM$ is essentially a ``surface'' 
\begin{eqnarray}
&\Gamma:&[a, b]\times [c, d]\rightarrow M\nonumber\\
&&(s, t)\mapsto \Gamma(s,t)\nonumber
\end{eqnarray}
on $M$, we can talk about ``transverse'' paths $\Gamma_t:[a, b]\rightarrow M$ and ``longitudinal'' paths  $\Gamma^s:[c, d]\rightarrow M$. Then
we may consider ``sideways'' composition and ``top-bottom'' composition for such ``surfaces''. On the other hand the essential idea of a double category ${\mathcal C}_{(2)}$ over a category ${\mathcal C}$ is, objects of ${\mathcal C}_{(2)}$ are morphisms ($1$-morphisms) of ${\mathcal C}$ and 
morphisms ($2$-morphisms) in ${\mathcal C}_{(2)}$ are equipped with two types of composition laws (namely, ``horizontal'' and ``vertical''). 
So, if we take into account that two elements of $\mpM$  might be composable as paths on $M$, then it indicates that  compositions of such `surfaces' (given as paths on a path space) correspond to compositions of $2$-morphisms in a higher category. In section~\ref{s:double} we show that if $M$ is complete then the geodesics on path space $\mpM$ induces a double category structure. Here, basically a $0$-morphism (object) consists of a triplet, a  point in $M,$ a tangent vector and an element of $\mathbb R$. Whereas a $1$-morphism is a triplet, given by a path on $M,$ a tangent vector along the path and an element of $\mathbb R$. Finally, a $2$-morphism is specified by a geodesic on the path space and an open interval in $\mathbb R$. Before we could make these ideas mathematically precise, we need to settle few technical difficulties. That has been done in section~\ref{s:double}. We use the technique of ``{\em back-track equivalence}'', which ensures that the equivalent set of back-track equivalent paths on $M$ can be treated 
as a morphism in a category, whose  object space is $M$. So, in section~\ref{s:double} we first discuss the notion of back-track equivalence
and construct a category ${\mathbb P}^{bt}$, whose object set is $M$ and morphisms are (back-track equivalent) paths on $M$. We show in Proposition~\ref{pr:equiva} that geodesics on $\mpM$ preserves the back-track equivalence identification on $\mpM$.  Theorem~\ref{th:double} proves the existence of a double category, whose base category is defined by ${\mathbb P}^{bt}$ (with some additional factors).

We end this paper with  a physical interpretation of the categories obtained in section~\ref{s:double}. In particular, we show that the morphisms of double  
category  in Theorem~\ref{th:double} can be interpreted as the world sheets generated by free strings (without any external force) on the Riemannian manifold $M$.


\section{Metric and covariant connection on path space}\label{s:vectcon}
Let $M$ be a Riemannian manifold. We define  path space $\mpM$ as the space of all smooth maps $\gamma:[0,1]\rightarrow M$ defined on an open 
interval $[0,1]$.  We denote the evaluation map as ${\rm ev}_t$, 
\begin{equation}
{\rm ev}_t:{\mathcal P}M \rightarrow M:\gamma \mapsto {\rm ev}_t(\gamma)=\gamma(t), \forall t\in [0,1]\label{eva}. 
\end{equation}
For a $\gamma \in \mpM$  we  get a tangent vector at $\gamma$ to be the differential of the map ${\rm ev}_t,$
$$X:=\{X:[0,1]\to TM,\text{  }  \text{smooth vector field along } \gamma \}.$$ 

Let $g$ be a metric on the manifold $M$. It defines an $L^{2}$ metric ${\widetilde g}$ on $\mpM$ ~\cite{wurzbacher, biswas-chat} given by
\begin{equation}
\left(\widetilde g (K_1,K_2)\right)_{\gamma}:=\int_{\gamma} g_{\gamma(t)}(K_1(t),K_2(t))dt,\label{tildeh},
\end{equation}
where $K_1,K_2$ are vector fields on $\mpM$. Much of the content of this section could be found in~\cite{freed}. We will briefly 
recall some results for our purpose. Let $\nabla$ be a covariant connection on $M$. Then define a connection on $\mpM$ by point-wise evaluation:
\begin{equation}\label{connection}
(\widetilde \nabla_{X} (Y))(t):=\nabla_{X(t)}Y(t),
\end{equation}
where $X, Y$ are vector fields on $\mpM.$ Thus we have the following proposition:
\begin{prop}\label{prop:metriccompati}
If  metric $g$ is compatible with the connection $\nabla$ on $M$, then so is metric $\widetilde g$  
 with $\widetilde{\nabla}$ on $\mpM$. 
\end{prop}
\begin{proof}
 The proof follows by verifying the famous ``six terms'' formula (Theorem 2.2 and 
Proposition 2.3, Chapter-IV \cite{kob1})
\begin{eqnarray}
2{\widetilde g}({\widetilde \nabla}_X Y, Z):=&&\iota_{X}{\mbox d}{\widetilde g}(Y,Z)+\iota_{Y}{\mbox d}{\widetilde g}(X,Z)-\iota_{Z}{\mbox d}{\widetilde g}(X,Y)\nonumber\\
&&+{\widetilde g}([X,Y],Z)+{\widetilde g}([Z,X],Y)+{\widetilde g}(X,[Z,Y]),\label{covpath}
\end{eqnarray}
where $X,Y,Z$ are vector fields on $\mpM$  and $\iota$ is the contraction. 
\end{proof}
We define a path on the path space $\mpM$ by a continuous map
\begin{equation}
\Gamma: [a,b]\rightarrow \mpM;  \text{  }s\mapsto \Gamma(s)\in \mpM.\label{path}
\end{equation}
Thus for each $s\in [a,b]$ we have a path given by
$$\Gamma(s)(t):=\Gamma_s(t):=\Gamma(s,t).$$ 
We denote `longitudinal' and `transverse'  paths respectively as
\begin{eqnarray}
&&\Gamma^s:[0,1]\rightarrow M, \Gamma^{s}(t)=\Gamma(s,t)\\\label{long}
&&\Gamma_t:[a,b]\rightarrow M, \Gamma_{t}(s)=\Gamma(s,t).\label{trans}
\end{eqnarray}
A tangent vector field along this path $\Gamma$ is given by
\begin{equation}
\Gamma':[a, b]\to T(\mpM);\text{ }s\mapsto \frac{\partial}{\partial s} \Gamma(s,t)\label{tgt vectorfield}
\end{equation}
As a consequence of the point-wise definition of our covariant derivative in \eqref{connection}, it is obvious that:
\begin{prop}\label{prop:parapath}
Let $[a, b]\subset \mathbb R$ be an interval containing $0$ and $\Gamma:[a,b]\rightarrow \mpM$ a path on path space $\mpM$. 
If a vector $V \in T_{\Gamma(0)}\mpM $ is given by 
$V(t)\in T_{\Gamma(0)(t)}M \equiv T_{\Gamma^0(t)}M $. Then parallel transport 
of $V$ along $\Gamma:[a,b]\rightarrow \mpM$ by the Riemannian connection on path space 
defined above is given by the solution of $$\nabla_{{\Gamma'}_t(s)}{X_t(s)}=0, \qquad \textit{{for each}} \hskip 0.2cm  t\in [0,1],$$
with the initial condition $X_t(0)=V(t)$, for each $t\in [0,1].$
\end{prop}
Also $\Gamma$ is a geodesic on $\mpM,$ if and only if each transverse path $\Gamma_t,$
as defined in \eqref{trans}, is a geodesic on $M$ for each $t \in[0, 1]$. Thus we have the following proposition [Proposition 3.1 \cite{mishra3}]
\begin{prop}\label{prop:initialcongeo}
For any given $\gamma\in \mpM$ and any vector $V\in T_{\gamma}\mpM$, there is a
unique path space geodesic $\Gamma:[a,b]\rightarrow \mpM$, such that
$\Gamma(0)=\gamma$ and $\Gamma'(0)=V$, where $[a, b]$ is an interval containing $0$.
\end{prop}
Recall that a linear connection on $M$ is complete if for any $p\in M$ and $X\in T_p M$, the geodesic 
$\gamma:[a,b]\rightarrow M$ with the initial conditions $\gamma(0)=p, {\dot{\gamma}}(0)=X$ can be extended for 
all values of $t$, i.e. $\gamma$ can be defined as $\gamma:(-\infty,\infty)\rightarrow M$. 
From Proposition~\ref{prop:initialcongeo} it follows that:
\begin{prop}\label{prop:completeness}
If $M$ is complete with respect to a connection $\nabla$, then $\mpM$ is
complete with respect to the induced connection $\widetilde{\nabla}.$
\end{prop}
Now suppose for each $V\in T_{\gamma}\mpM$, ${}^V\Gamma$ be the geodesic with the initial condition ${}^V\Gamma(0)=\gamma, {{}^V\Gamma}^{'}(0)=V$, then 
exponential map on the path space $\rm Exp$ is 
\begin{eqnarray}
&&{\rm Exp}:T_{\gamma}\mpM \rightarrow \mpM\nonumber\\
&&{\rm Exp}:V \mapsto {}^V\Gamma(1)\label{expo}.
\end{eqnarray}
As a consequence of Proposition~\ref{prop:initialcongeo},  $({\rm Exp}(V))(t)$ is given by
\begin{equation}
{\rm exp}:V(t) \mapsto {{}^{V(t)}}\Gamma_t(1), \qquad {\textit{for each}}\hskip 0.2cm t\in [0,1]\label{expoM},
\end{equation}
where ${}^{V(t)}\Gamma_t$ is the geodesic on $M$ with  initial conditions ${}^{V(t)}\Gamma_t(0)=\gamma(t), {{}^{V(t)}\Gamma}_{t}^{\prime}(0)=V(t)$
and ${\rm exp}$ is the exponential map on $M$. It is obvious from \eqref{expoM} that, if ${\rm exp}$ is defined on the entire $T_p M$ for each $p\in M$, then ${\rm Exp}$ is also defined on the entire $T_{\gamma}\mpM$ for each $\gamma\in \mpM$. This exponential map on $\mpM$ provides a chart on  $\mpM$~\cite{mishra, mishra2}.


\section{Distance function on $\mpM$}\label{s:metric}
Let $\gamma:[0,1]\rightarrow M$ be a path on $M$, then energy functional evaluated at $\gamma$ is defined as
\begin{equation}
{\mathcal E}_{\gamma}:=\frac{1}{2}\int_{0}^{1}{{(g (\gamma'(s),\gamma'(s)))_{\gamma(s)}}}ds.\label{energyfuncor}
\end{equation}
Suppose a path on path space $\Gamma:[a,b]\rightarrow \mpM$ is given, then we define energy functional evaluated at $\Gamma$ on path space to be,
\begin{equation}
E_{\Gamma}:=\frac{1}{2}\int_{a}^{b}{{({\widetilde g} (\Gamma'(s),\Gamma'(s)))_{\Gamma(s)}}}ds.\label{energyfuncpath}
\end{equation}
Therefore we can write \eqref{energyfuncpath} as
\begin{equation}
E_{\Gamma}=\frac{1}{2}\int_{a}^{b}\left(\int_{0}^{1} g\left(\Gamma'(s,t),\Gamma'(s,t)\right)dt\right)ds.\label{arclenthpathwithg}
\end{equation}
As the integrations with respect to
$s$ and $t$ are independent in \eqref{arclenthpathwithg}, we have
\begin{eqnarray}
&&E_{\Gamma}=\frac{1}{2}\int_{0}^{1}\left(\int_{a}^{b} g\left(\Gamma'(s,t),\Gamma'(s,t)\right)ds\right)dt,\nonumber\\
&&\Rightarrow E_{\Gamma}=\int_{0}^{1}{\mathcal E}_{\Gamma_{t}}dt,\label{patheore}
\end{eqnarray}
where $\Gamma_{t}:[a,b]\rightarrow M$ for each $t\in[0,1],$ as defined in \eqref{trans}.
The energy functional on path spaces has been described in ~\cite{Docarmo, grove}.
 Let us define
\begin{eqnarray}
{\widetilde {\rm d}}(\gamma_0,\gamma_1)&&:={\rm {infimum {\hskip 0.1cm}  of}}\hskip 0.2cm \sqrt{2|b-a|}\sqrt{E_{\Gamma}} {\rm for {\hskip 0.1cm}  all} {\hskip 0.1cm} \{\Gamma:[a,b]\rightarrow \mpM|\Gamma(a)=\gamma_{0},\Gamma(b)=\gamma_1\}\nonumber\\
&=&{\rm {infimum {\hskip 0.1cm}  of}}\hskip 0.2cm \sqrt{2|b-a|}\sqrt{\int_{0}^{1}{\mathcal E}_{\Gamma_{t}}dt}\nonumber\\
&&{\rm {for {\hskip 0.1cm}  all}} {\hskip 0.1cm} \{\Gamma:[a,b]\rightarrow \mpM|\Gamma(a)=\gamma_{0},\Gamma(b)=\gamma_1\}.\label{d2}
\end{eqnarray}
It can be easily verified that ${\widetilde {\rm d}}$ is  a well defined distance function.
 Recall the exponential map ${\rm Exp}$ on $\mpM$ is given by \eqref{expoM}
$$X(t)\mapsto {\rm exp}_{\gamma(t)}X(t),$$
where $X\in T_{\gamma}\mpM$ and ${\rm exp}$ is the exponential map on $M$. For our Rimannian connection
${\widetilde \nabla}$ this exponential map explicitly reads
\begin{eqnarray}
&&{\rm Exp}:V \mapsto {{}^{V}}\Gamma(1), \label{expopM},\\
&&{\rm exp}:V(t) \mapsto {{}^{V(t)}}\Gamma_t(1), \qquad {\textit{for each}}\hskip 0.2cm t\in [0,1]\label{expoM2}
\end{eqnarray}
where ${}^{V(t)}\Gamma_t$ is the geodesic on $M$ with the initial conditions
${}^{V(t)}\Gamma_t(0)=\gamma(t), {{}^{V(t)}\Gamma}_{t}^{\prime}(0)=V(t)$
and, by Proposition~\ref{prop:initialcongeo} ${}^{V}\Gamma$, is the corresponding
geodesic on $\mpM$. Hence it follows that,
\begin{prop}\label{prop:existencenormalne}
If ${\mathcal U}_{\gamma_0}\subset \mpM$ is the normal neighbourhood on $\mpM$ around $\gamma_0\in \mpM$, then
\begin{equation}
{\mathcal U}_{\gamma_0}=\{\gamma \in \mpM|\gamma(t)\in U_{\gamma_0(t)}, \textit{for each}\hskip0.2cm t\in [0,1]\},\label{nornrionpath}
\end{equation}
where $U_{\gamma_{0}(t)}\subset M$ is a normal neighbourhood around $\gamma_{0}(t).$
\end{prop}
Rest of this section would be devoted to prove the following theorem
\begin{theorem}\label{theo:normalneighbour}
Any $\gamma_1, \gamma_2\in {\mathcal U}_{{\gamma}_0}$ can be joined  by a unique path space geodesic and
length of that geodesic is ${\widetilde {\rm d}}(\gamma_1,\gamma_2)$.  ${\mathcal U}_{{\gamma}_0}$ is convex with respect to
the distance function ${\widetilde {\rm d}}$.
\end{theorem}
We proceed with the following proposition:
\begin{prop}\label{prop:joininggeoinnormal}
Every  $\gamma_1, \gamma_2\in {\mathcal U}_{\gamma_0}$ can be joined  by a unique path space geodesic
lying in ${\mathcal U}_{\gamma_0}$.
\end{prop}
\begin{proof} We recall that for a $C^{\infty}$ manifold $M$ with an affine connection, there always
exists an open neighbourhood $N_{p}$ of the zero vector $0\in T_pM$, such that
\begin{enumerate}
\item The exponential map ${\rm exp}:N_{p}\rightarrow U_p$ is diffeomorphic.
\item if $X\in N_p$, then $sX\in N_p$, for some interval $[a,b]$ and $s\in[a,b],$
\end{enumerate}
where $U_p\subset M$ is an open normal neighbourhood around $p$. We set the exponential map ${\rm exp}$ such that,
\begin{equation}
{\rm exp}(sX):=\gamma^{geo}(s),\label{set}
\end{equation}
where $X\in N_p$ and $\gamma^{geo}$ is the unique geodesic with the initial
conditions $p\in M,X\in T_pM$. Now consider  arbitrary $\gamma_1, \gamma_2\in {\mathcal U}_{\gamma_0}$,
 then by \eqref{nornrionpath}, for any $t\in [0,1]$, $\gamma_1(t), \gamma_2(t)\in U_{\gamma_{0}(t)}$.
Hence for each $t\in[0,1],$ $\gamma_1(t)$ can be joined to $\gamma_{2}(t)$ by a unique geodesic, say
$\gamma_{t}^{geo}(s)$. So we have a path on path space joining $\gamma_1$ and $\gamma_2$, with the following properties
\begin{eqnarray}
\Gamma:[a,b]\times [0,1]\rightarrow M\\
\Gamma:(s,t)\mapsto\gamma_{t}^{geo}(s)\label{geomapstogeo}.
\end{eqnarray}
Moreover this path on path space $\Gamma$ is such that each path $\Gamma_t=\gamma^{geo}_t$ is a geodesic on $M,$
with starting point $\gamma_{1}(t)\in M$, with some `velocity' $V_t\in T_{\gamma_{1}(t)}M$. Hence from
Proposition~\ref{prop:initialcongeo} we conclude $\Gamma$ to be the unique geodesic starting from $\gamma_{1}\in \mpM$ with `velocity' $V\in T_{\gamma_1} \mpM$. Hence any $\gamma_1\in {\mathcal U}_{\gamma_0}$ can be joined to $\gamma_2$ by a unique geodesic.
\end{proof}
 
\begin{prop}\label{prop:equitopo}
Suppose  a geodesic $\Gamma^{geo}:[a,b]:\rightarrow \mpM$ exists between $\gamma_1$ and $\gamma_2$. Assume each geodesic
$\Gamma^{geo}_t:[a,b]\rightarrow M$  on $M$ to be minimizing, then length of $\Gamma^{geo}$ is given by
$$
L{(\Gamma^{geo})}={\widetilde {\rm d}}(\gamma_1,\gamma_2),
$$
and hence $\Gamma^{geo}$ is minimizing.
\end{prop}
\begin{proof}Each $\Gamma^{geo}_t:[a,b]\rightarrow M$ is a geodesic between $\gamma_{1}(t)$
and $\gamma_2(t)$ for each $t\in [0,1]$. From Cauchy-Schwarz inequality, for any $\Gamma_t:[a,b]\rightarrow M,$ we have
\begin{equation}
\left({\mathcal L}(\Gamma_t)\right)^2\leq 2|b-a|{\mathcal E}_{\Gamma_t},\label{cauchyM}
\end{equation}
where ${\mathcal L}(\Gamma_t)$ is the arc length of $\Gamma_t$. The equality holds only for a geodesic. Hence
\begin{equation}
\left({\mathcal L}(\Gamma^{geo}_t)\right)^2 = 2|b-a|{\mathcal E}_{\Gamma^{geo}_t}.\label{cauchyMgeo}
\end{equation}
For minimizing geodesics, ${\rm d}(\gamma_1(t),\gamma_2(t))={\mathcal L}(\Gamma^{geo}_t),$ where ${\rm d}$ is the distance function on $M$,
 and since each $\Gamma^{geo}_{t}$ is minimizing, 
as a consequence of  \eqref{d2} and \eqref{patheore} it follows that
\begin{equation}
\left({\widetilde {\rm d}}(\gamma_1,\gamma_2)\right)^2=2|b-a|\int_{0}^{1}{\mathcal E}_{\Gamma^{geo}_{t}}dt\label{d2expli}
\end{equation}
and from \eqref{cauchyMgeo}
\begin{equation}\label{d2expliind}
\begin{split}
&\left({\widetilde {\rm d}}(\gamma_1,\gamma_2)\right)^2=\int_{0}^{1}\left({\mathcal L}(\Gamma^{geo}_t)\right)^2dt,\\
&\Rightarrow \left({\widetilde {\rm d}}(\gamma_1,\gamma_2)\right)^2=\int_{0}^{1}\left({\rm d}(\gamma_1(t),\gamma_2(t))\right)^2dt.
\end{split}
\end{equation}
Now applying Cauchy-Schwarz inequality on the path space, we have
$$\left(L({\Gamma})\right)^2 \leq 2|b-a| E_{\Gamma}.$$
If $\Gamma$ is a geodesic, the equality holds. Hence
\begin{equation}
\left(L{(\Gamma^{geo})}\right)^2 = 2|b-a| E_{\Gamma^{geo}}.\label{geocauchypath}
\end{equation}
So from \eqref{d2expli}, \eqref{d2expliind} we get
\begin{equation}
\left(L{(\Gamma^{geo})}\right)^2 =\left({\widetilde {\rm d}}(\gamma_1,\gamma_2)\right)^2.\label{geolength}
\end{equation}
\end{proof}
On the other hand, we have seen in Proposition~\ref{prop:joininggeoinnormal} that any $\gamma_1, \gamma_2\in {\mathcal U}_{\gamma_0}$ can be
joined by a unique path space geodesic $\Gamma^{geo}$ and hence, each $\Gamma_t^{geo}$ is a geodesic 
between $\gamma_{1}(t)$ and $\gamma_2(t)$. But from Proposition~\ref{prop:existencenormalne} we know, if $\gamma\in {\mathcal U}_{\gamma_0}$ then
$\gamma(t)\in U_{\gamma_0(t)}$, for each $t$. Thus for each $t$, $\gamma_1(t)$ can be joined with $\gamma_2(t)$ by a unique minimizing geodesic and finally
according to Proposition~\ref{prop:equitopo} that gives a unique minimizing path space geodesic lying in ${\mathcal U}_{\gamma_0}$. Hence
\begin {cor}\label{cor:minigeo}
Any $\gamma_1, \gamma_2\in {\mathcal U}_{\gamma_0}$ can be joined by a minimizing path space geodesic lying in ${\mathcal U}_{\gamma_0}$. Thus 
${\mathcal U}_{\gamma_0}$ is convex.
\end{cor}
This completes the proof of Theorem~\ref{theo:normalneighbour}.

\section{Double category of the geodesics on the path space}\label{s:double}
Let $M$ be a Riemannian manifold and $\mpM$ be the space of $C^{\infty}([0, 1], M)$ maps. Consider the set of $C^{\infty}([0, 1], M)$ maps 
which are constants near the end points.  Below we will provide a precise definition of the same. We denote such a space as 
$$\mpM_{c}\subset \mpM.$$
We say  a path is constant near the end points, when there exists some $\delta >0$ such that for  $t_0\in [0,1]$ and $\gamma\in \mpM$,
the maps $\gamma|_{[0,t_0]}$ and $\gamma|_{[t_0, 1]}$ are constant maps whenever $t_0<\delta$ or $1-t_0<\delta.$
The purpose of introducing such a condition is to ensure that `composition' of two smooth paths remain a smooth path. Let  
$\gamma_1, \gamma_2\in \mpM$ and $\gamma_2(0)=\gamma_1(1)$, then by the composed path $\gamma_2\circ \gamma_1$ we mean 
\begin{eqnarray}
&\left(\gamma_2\circ\gamma_1\right)(t)&=\gamma_1(2t), \qquad 0\leq t\leq \frac{1}{2}\nonumber\\
&&=\gamma_2(2t-1),\qquad  \frac{1}{2}<t\leq 1.\nonumber
\end{eqnarray}

Now, we will impose an equivalence relation on $\mpM_c$, namely {\em back-track equivalence}. 
We refer to \cite{levy, saikat4} for a detail discussion on the topic. Note there exists a similar, but slightly  more 
general notion of equivalence under {\em thin homotopy}\cite{brown}, which we will not discuss here. Roughly two paths $\gamma_1$, $\gamma_2$
are back-track equivalent if there exists a path $\gamma_0$ such that
$$\gamma_1\circ (\gamma_0 \circ \gamma_0^{-1})=\gamma_2.$$
Here and onwards the reverse of  a map $\lambda:[a, b]\rightarrow M$ is given by
\begin{eqnarray}
&&\lambda^{-1}:[a, b]\rightarrow M, \\\nonumber
&&\lambda^{-1}(t_0):=\gamma(b+a-t_0), \qquad t_0\in [a, b]\nonumber.
\end{eqnarray}
Let us make the statement more precise. A path $\gamma:[0, 1]\rightarrow M$ is said to be back-tracked
over $[T, T+\sigma]$, where $[T, T+2\sigma]\subset [0, 1]$, if 
\begin{equation}\label{backtrack}
\gamma(T+u)=\gamma(T+2\sigma-u), \qquad \forall u\in [0, \sigma],
\end{equation}
and, by {\em back-track erasing} the portion $[T, T+\sigma],$  we obtain the map:
$$[0, 1-2\sigma]\rightarrow M$$
given by 
\begin{equation}\label{baera}
\begin{split}
t\mapsto &\gamma(t), \qquad t\in [0, T]\\
&\gamma(t-2\sigma)\qquad t\in [T+2\sigma, 1].
\end{split}
\end{equation}
Let us identify two paths $\gamma_1, \gamma_2$ under reparametrization; that is if there exists a strictly increasing map 
$\phi:[0,1]\rightarrow [0,1], \phi(0)=0, \phi(1)=1$ and
$\gamma_1=\gamma_2\phi,$ then $\gamma_1, \gamma_2$ are equivalent. Now, we define two paths $\gamma_1, \gamma_2$ to be 
{\em elementary back-track equivalent}, if there are $C^{\infty}$ maps 
\begin{equation}\label{lammaps}
\begin{split}
&\lambda_3:[0, T]\rightarrow M, \\
&\lambda_2:[T, T+\sigma] \rightarrow M, \\
&\lambda_1:[T+2\sigma, 1] \rightarrow M, 
\end{split}
\end{equation}
 such that
\begin{equation}\label{bacequ}
\begin{split}
&\gamma_1\phi_1=\lambda_1\circ \lambda_2\circ{\lambda_2}^{-1}\circ \lambda_3,\\
&\gamma_2\phi_2=\lambda_1\circ \lambda_3,
\end{split}
\end{equation}
for some strictly increasing $\phi_1:[0, 1]\rightarrow [0, 1], \qquad \phi_1(0)=0, \phi_1(1)=1$ and $\phi_2:[0, T-2\sigma]\rightarrow [0, 1], \qquad \phi_2(0)=0, \phi_2(T-2\sigma)=1.$ \eqref{bacequ} can be stated as,  $\gamma_1$ is obtained from $\gamma_2$ by erasing the back-track part $\lambda_2\circ \lambda_2^{-1}.$ 
Now, if there is a sequence of paths $\gamma_1,\gamma_2, \cdots, \gamma_n$ such that  $\gamma_i$ is elementary back-track equivalent to $\gamma_{i+1},i=[1, n-1],$
then we call $\gamma_1, \gamma_n$ to be {\em back-track equivalent}. We denote it as
$$\gamma_1\simeq_{bt}\gamma_n.$$
It can be shown back-track equivalence has following properties~\cite{saikat3}
\begin{itemize}
\item {The back-track equivalence is preserved under reparametrization.}
\item {If $\gamma_1\simeq_{bt}\gamma_2$, ${\tilde \gamma}_1\simeq_{bt}{\tilde \gamma}_2$ and $\gamma_1, {\tilde \gamma}_1$ are composable, then so is  
$\gamma_2, {\tilde \gamma}_2$, more over in that case 
\begin{equation}\label{backcompo}
\gamma_1\circ {\tilde \gamma}_1\simeq_{bt} \gamma_2\circ {\tilde \gamma}_2.
\end{equation} }
\end{itemize}
Now, define the quotient space under the back-track equivalence relation:
\begin{equation}\label{btspace}
{\mpM}^{bt}_c:=\mpM_c/{\simeq_{bt}}.
\end{equation}
We will not notationally distinguish between elements of ${\mpM}^{bt}_c$ and $\mpM_c$, that is $\gamma \in {\mpM}^{bt}_c$ would actually mean
the equivalence class $[\gamma]_{\simeq_{bt}}.$

Recall a tangent vector $X\in T_{\gamma}\mpM$ is given by a smooth vector field  $X(t)\in T_{\gamma(t)}M, t\in [0, 1]$ along $\gamma$.
\begin{enumerate}
\item{We define a vector $X\in T_\gamma \mpM_c$ to be a vector field along $\gamma$ such that it is constant near the end points $0, 1.$ }

\item {We define a vector $X\in T_{\gamma}\mpM^{bt}_{c}$ to be a vector field along the path $\gamma$ which has the property (1) and
 back-track of $\gamma$ coincides with that of $X.$ That is, if $\gamma$ has a back-track in $[T, T+\sigma]$ as defined in \eqref{backtrack}, then 
\begin{equation}\label{backtrackvect}
X(T+u)=X(T+2\sigma-u), \qquad \forall u\in [0, \sigma].
\end{equation}
}
\end{enumerate}
We have seen in Proposition~\ref{prop:initialcongeo} that given a $ \mpM\ni \gamma:[0, 1]\rightarrow M$ and 
a $T_{\gamma}\mpM\ni v:[0,1]\rightarrow T_{{\gamma}(t)}M$ we have a unique 
geodesic 
\begin{eqnarray}
&\Gamma^{geo}&:[a,b]\rightarrow \mpM\nonumber\\
&&[a,b]\times [0,1]\rightarrow M\nonumber
\end{eqnarray}
on the path space. This unique geodesic has following description:
Each transverse path $\Gamma^{geo}_t:[a, b]\rightarrow M, t\in [0,1]$ is a geodesic with initial conditions  
\begin{itemize}
\item[(i)] $\Gamma^{geo}_t(0)=\gamma(t)$
\item[(ii)] $\frac {\partial \Gamma^{geo}_t(s)}{\partial s}|_0=v(t).$
\end{itemize}
Thus we infer:
\begin{prop}\label{pr:bactrageo}
Suppose $\gamma\in \mpM_c$ and $T_{\gamma}\mpM_c\ni v:[0,1]\rightarrow T_{{\gamma}(t)}M$ have back-tracking 
in $[T, T+\sigma],$ and $\Gamma:[a, b]\rightarrow \mpM$ is the unique geodesic with initial conditions
$\gamma, v$. Then the longitudinal path defined by $$\Gamma^s:[0,1]\rightarrow M, \Gamma^{s}(t)=\Gamma(s, t)$$
satisfies
\begin{enumerate}
\item  {for each $s\in [a, b],$ $\Gamma^s$ has the back-tracking in $[T, T+\sigma]$}
 and 
\item {$\Gamma^s\in \mpM_c$ for each $s\in[a, b]$}.
\end{enumerate}
\end{prop}
\begin{proof}
By Proposition~\ref{prop:initialcongeo} each transverse path $\Gamma_t:[a, b]\rightarrow M, t\in [0,1]$ is a 
geodesic with initial conditions $(\gamma(t), v(t)).$ 
\begin{enumerate}
\item{Since $\gamma$ and $v$ have back-track in $[T, T+\sigma]$, by \eqref{backtrack}
\begin{eqnarray}
&\gamma(T+u)=\gamma(T+2\sigma-u), \qquad \forall u\in [0, \sigma],\nonumber\\
&v(T+u)=v(T+2\sigma-u), \qquad \forall u\in [0, \sigma],\nonumber
\end{eqnarray}
and since each geodesic $\Gamma_t:[a, b]\rightarrow M, t\in [0,1]$ is 
 uniquely determined by the initial conditions $\gamma(t), v(t),$
we have same back-track for the paths $\Gamma^{s}:[0,1]\rightarrow M$.}
\item{follows from similar argument.}
\end{enumerate}
\end{proof}
 Proposition~\ref{pr:bactrageo} essentially states that a back-tracking is mapped to a back-tracking under the exponential map
${\rm Exp}.$ Thus we have the following corollary.
\begin{cor}\label{cor:backcompo}
Suppose $\gamma \in \mpM_c$ is obtained by back-track erasing the portion $\gamma_0$ from $\tilde \gamma$, that is there exists $\gamma_1, \gamma_2$ such that 
$\gamma=\gamma_1\circ \gamma_2$ and $\tilde \gamma=\gamma_1\circ\gamma_0\circ \gamma_0^{-1}\circ \gamma_2.$ Let $X\in T_{\gamma}\mpM_c$ 
be obtained by identifying with $X_1$ for the first half and with $X_2$ with the second half, where $X_1, X_2$ are restrictions of the vector field along 
the path $\tilde \gamma$ on $\gamma_1, \gamma_2$ respectively. Let the geodesic obtained from  the initial condition 
$(\gamma_i, X_i), i=1, 2$ is ${}^{i}\Gamma,$ where $X_i$ is the vector field obtained by restricting to the portion $\gamma_i, i=1, 2.$ Let
 ${}^{i}\Gamma^s:[0,1]\rightarrow M, {}^i\Gamma^{s}(t)={}^i\Gamma(s, t)$ be the longitudinal path defined for each $s\in [a, b]$. Then
\begin{equation}\label{backcom}
\Gamma^s={}^1\Gamma^s\circ {}^2\Gamma^s,
\end{equation}
where $\Gamma$ is the geodesic obtained from the initial condition $(\gamma, X).$
\end{cor} 
Proposition~\ref{pr:bactrageo} and Corollary~\ref{cor:backcompo}  imply
\begin{prop}\label{pr:equiva}
Suppose $\gamma_1, \gamma_2\in \mpM_c$,$\gamma_1\simeq_{bt}\gamma_2$ and $X_1\simeq_{bt} X_2$. Let the geodesic obtained from the the initial conditions 
$(\gamma_i, X_i), i=1, 2$ is ${}^{i}\Gamma,$ then 
$${}^{1}\Gamma^s\simeq_{bt}{}^{2}\Gamma^s.$$
\end{prop}
Thus if $[\gamma]_{bt}\in \mpM_{c}^{bt}$ and $[X]_{bt}\in T_{\gamma}\mpM_{c}^{bt}$, then we have a $[\Gamma^s]\in \mpM_{c}^{bt}$ for each $s\in [a,b],$ where $\Gamma$ is the geodesic obtained from initial conditions $(\gamma, X),$ (here, to make the distinction clear we write $[\gamma]_{bt}$ and $[X]_{bt}$).
Now we can define a category 
$${\mathbb P}^{bt},$$ 
whose objects are points of $M$ and  a morphism is given by
$\gamma\in \mpM_{c}^{bt}$ with source $\gamma(0)$ and target $\gamma(1)$ and composition is given by
$\gamma_2\circ \gamma_1$ (which is well defined by \eqref{backcompo}) and identity morphism at $m\in M$ is the constant path. Note, since 
$\gamma_1, \gamma_2$ are constant maps near the end points $\gamma_2\circ \gamma_1$ is $C^{\infty}$. Also, the general  concatenation of paths is not associative. 
However,  back-track equivalence makes the composition associative (see section 3 and 6 of \cite{saikat4}).  Thus, everything is well defined here.

Next we show geodesics on the path space of a complete Riemannian manifold naturally define a double category. First let us specify 
what we mean by a double category (terminology varies in the literature.) 

By a {\em double category} ${\mathcal C}_{(2)}$ over a category ${\mathcal C},$ we understand a category whose 
objects are the arrows of ${\mathcal C}$ and on which there is a partially-defined binary operation
$$(G,F)\mapsto G\circ_H F$$
 for certain pairs of morphisms $F, G\in {\rm Mor}({\mathcal C}_{(2)})$, satisfying:
\begin{itemize}
\item[(i)] $s(G\circ_H  F)={\rm s}(G)\circ {\rm s}(F)$ and ${\rm t}(G\circ_H F)={\rm t}(G)\circ {\rm t}(F)$, whenever $G\circ_H F$ is defined;
\item[(ii)] the exchange law
$$(G'\circ G)\circ_H (F'\circ F)=(G'\circ_H F')\circ (G\circ_HF),$$
holds whenever either side is defined, where $\rm {s, t}$ are source and target maps respectively.
\end{itemize}

 Assume $M$ to be complete. Then, by Proposition~\ref{prop:completeness} $\mpM$ is also complete. 
Now let us define a category 
\begin{equation}\label{cat1}
{\rm \mathbf C}^{\rm geod},
\end{equation}
which has following description. An object in ${\rm \mathbf C}^{\rm geod}$ is given by a triplet, $(p, X, a),$ of
 a point in $p\in M$, a tangent vector $X\in T_{p}M$, and an element $a\in \mathbb R$. A morphism is specified by another triplet, $(\gamma, \tilde X, a)$, of a path $\gamma\in \mpM_c^{bt}$, a vector field $\tilde X\in T_{\gamma}\mpM_c^{bt}$,  an element $a\in \mathbb R$. The source and target of a morphism 
$f=(\gamma, \tilde X, a)\in {\rm Mor}({\rm \mathbf C}^{\rm geod})$
are respectively given by
\begin{eqnarray}
{\rm s}(\gamma, \tilde X, a)=(\gamma(0), {\tilde X}(0), a),\hskip 0.1 cm {\rm and} \hskip 0.1 cm {\rm t}(\gamma, \tilde X, a)=(\gamma(1), {\tilde X}(1), a)\label{st1}
\end{eqnarray}
and the composition reads
\begin{equation}
(\gamma_2, \tilde X_2, a)\circ (\gamma_1, \tilde X_1, a):=(\gamma_2\circ \gamma_1, \tilde X_2\circ \tilde X_1, a),\label{1comp}
\end{equation}
where $\gamma_2\circ \gamma_1$ is the composition in the category ${\mathbb P}^{bt}$ and $\tilde X_2\circ \tilde X_1\in T{\gamma_2\circ \gamma_1}\mpM$
is the smooth vector field along $\gamma_2\circ \gamma_1,$ given by point wise identification with $\tilde X_1$ for the first half and $\tilde X_2$ for the second half.
 Note, as per the composition law in \eqref{1comp}, we must have the composability condition  $\gamma_2(0)=\gamma_1(1)$ and $\tilde X_2(0)=\tilde X_1(1).$
 That means we have (an equivalence class of) smooth  non degenerate $\tilde X_2\circ \tilde X_1$ along $\gamma_2\circ \gamma_1.$ The identity morphism $1_{p, X, a}$ corresponding to $(p, X, a)$ is simply the pair of constant maps $[0,1]\to p$, $[0,1]\to X$ and $a\in {\mathbb R}$. It can be verified that composition in \eqref{1comp} is associative~\cite{saikat4}.

By assumption $M$ is complete, and thus by Proposition~\ref{prop:completeness}, ${\mathcal P}M$ is also complete. Therefore, we may take any arbitrary interval $[a, b]\subset {\mathbb R}$ to define a geodesic segment.. Now let $\Gamma^{(\gamma, \tilde X)}$ be the geodesic on the path space obtained from the initial conditions $\gamma, \tilde X.$ We choose an interval $[a,b]$ and denote the geodesic segment on this interval by 
$${}_{[a,b]}\Gamma^{(\gamma, \tilde X)}.$$
Let us define following source-target maps respectively,
\begin{equation}\label{2st}
\begin{split}
&{\rm S}({}_{[a,b]}\Gamma^{(\gamma, \tilde X)}):=(\lambda_a, {\tilde Y}_a, a), \hskip 0.1 cm {\rm and}\hskip 0.1 cm  {\rm T}({}_{[a,b]}\Gamma^{(\gamma, \tilde X)}):=(\lambda_b, {\tilde Y}_b, b),\\
&{\rm where}\\
&\lambda_a:=\Gamma^{(\gamma, \tilde X)}(a)\\
&{\tilde Y}_a(t):=\frac{\partial \Gamma^{(\gamma, \tilde X)}_t(s)}{\partial s}|_a\\
&\lambda_b:=\Gamma^{(\gamma, \tilde X)}(b)\\
&{\tilde Y}_b(t):=\frac{\partial \Gamma^{(\gamma, \tilde X)}_t(s)}{\partial s}|_b
\end{split}
\end{equation}
Suppose $\gamma_1\simeq_{bt} {\tilde \gamma}_1$ and  $X_1\simeq_{bt} {\tilde X}_1$ , then by Proposition~\ref{pr:equiva}
$$\Gamma^s\simeq_{bt}{\tilde \Gamma}^s.$$
Hence we may as well assume $\gamma\in \mpM_c^{bt}$ and $X\in T_{\gamma}\mpM_c^{bt}.$ From now on we will always work assuming this back-track identification of paths.
 Now suppose ${}_{[b,c]}\Gamma^{(\gamma_2, \tilde X_2)}$ and ${}_{[a,b]}\Gamma^{(\gamma_1, \tilde X_1)}$ are two geodesic segments obtained from the respective initial conditions $\left(\gamma_2\in \mpM_c^{bt}, \tilde X_2\in T_{\gamma_2}\mpM_c^{bt}\right)$ and 
$\left(\gamma_1\in \mpM_c^{bt}, \tilde X_1\in T_{\gamma_1}\mpM_c^{bt}\right),$  defined on the intervals $[b, c]$ and $[a, b]$ respectively. Further assume, 
$${\rm S}({}_{[b,c]}\Gamma^{(\gamma_2, \tilde X_2)})={\rm T}({}_{[a,b]}\Gamma^{(\gamma_1, \tilde X_1)}).$$
For each $t\in[0,1]$ the above equation implies 
\begin{equation}
\Gamma_t^{(\gamma_2, \tilde X_2)}(b)= \Gamma_t^{(\gamma_1, \tilde X_1)}(b)\label{samepath}
\end{equation}
and 
\begin{equation}
\frac{\partial \Gamma_t^{(\gamma_2, \tilde X_2)}(s)}{\partial s}|_b =\frac{\partial  \Gamma_t^{(\gamma_1, \tilde X_1)}(s)}{\partial s}|_b \label{samedir}.
\end{equation}
Since $\Gamma_t^{(\gamma_1, \tilde X_1)}$ (respectively  $\Gamma_t^{(\gamma_2, \tilde X_2)}$ ) is a geodesic, $\frac{\partial  \Gamma_t^{(\gamma_1, \tilde X_1)}(s)}{\partial s}|_b$ (respectively  $\frac{\partial  \Gamma_t^{(\gamma_2, \tilde X_2)}(s)}{\partial s}|_b$) is parallel to $\Gamma_t^{(\gamma_1, \tilde X_1)}$ (respectively $\Gamma_t^{(\gamma_2, \tilde X_2)}$ ), thus by \eqref{samedir} the geodesic ${}_{[b,c]}\Gamma_t^{(\gamma_2, \tilde X_2)}$ is a geodesic in the same direction as ${}_{[a,b]}\Gamma_t^{(\gamma_2, \tilde X_2)}$. 
Now we can define a path segment in the interval $[a, c]$ as follows
\begin{eqnarray}
&\bigl({}_{[b,c]}\Gamma_t^{(\gamma_2, \tilde X_2)}\star {}_{[a,b]}\Gamma_t^{(\gamma_1, \tilde X_1)}\bigr)(s)&={}_{[a,b]}\Gamma_t^{(\gamma_1, \tilde X_1)}(s), \qquad a\leq s\leq b\nonumber\\
&&={}_{[b,c]}\Gamma_t^{(\gamma_2, \tilde X_2)}(s), \qquad b< s\leq c\label{starcom}.
\end{eqnarray}
But, from \eqref{samedir} it follows that above composition defines a geodesic segment in the interval $[a, c]$ with initial conditions $(\gamma_1(t), \tilde X_1(t))$, and since the relation holds for each $t\in [0,1],$ we have a necessary condition:
$$(\gamma_1, \tilde X_1)=(\gamma_2, \tilde X_2)$$
and therefore we have 
\begin{equation}
{}_{[b,c]}\Gamma_t^{(\gamma_1, \tilde X_1)}\star {}_{[a,b]}\Gamma_t^{(\gamma_1, \tilde X_1)}={}_{[a,c]}\Gamma_t^{(\gamma_1, \tilde X_1)}.\label{geocat}
\end{equation}
 In other words the composition  is just extension of the geodesic segment  
${}_{[a,b]}\Gamma_t^{(\gamma_1, \tilde X_1)}$ from the interval $[a,b]$ to 
$[a, c]$. So, we can define 
\begin{eqnarray}
&\bigl({}_{[b,c]}\Gamma^{(\gamma_1, \tilde X_1)}\star {}_{[a,b]}\Gamma^{(\gamma_1, \tilde X_1)}\bigr):&[a, c]\rightarrow \mpM\nonumber\\
&&[a, c]\ni s\mapsto {}_{[a,c]}\Gamma_t^{(\gamma_1, \tilde X_1)}(s).\label{2comp}
\end{eqnarray}
It is obvious that the above composition is associative. We define the identity morphism ${\mathbf 1}_{(\gamma, {\widetilde X}, a)}$ by 
$[a,a]\to \{\gamma\}$. Thus we have a category whose objects are given by $({\gamma, {\widetilde X}, a})$, a morphism is given by ${}_{[a,b]}\Gamma^{(\gamma_1, \tilde X_1)}$ with source-target given by \eqref{2st} and composition by \eqref{2comp}. We denote this category as
\begin{equation}\label{cat2}
{\mathbf C}^{\rm geod}_{(2)}.
\end{equation} 
The partial product $\star_{\rm H}$ is defined as follows. Consider ${}_{[a,b]}\Gamma^{(\gamma_1, \tilde X_1)}$ and  ${}_{[a,b]}\Gamma^{(\gamma_2, \widetilde X_2)}$, such that $\gamma_1(1)=\gamma_2(0)$ and $\tilde X_1(1)=\tilde X_2(0)$. Then, since each $\Gamma_t^{(\gamma_1, \tilde X_1)}$ (respectively $\Gamma_t^{(\gamma_2, \widetilde X_2)}$) is a geodesic uniquely determined by initial conditions $(\gamma_1(t), {\widetilde X}_1(t))$ (respectively $(\gamma_2(t), {\widetilde X}_2(t))$), we have
\begin{equation}
\begin{split}\label{hormat}
&\Gamma_1^{(\gamma_1, \tilde X_1)}=\Gamma_0^{(\gamma_2, \tilde X_2)}\\
\Rightarrow&{}_{[a,b]}\Gamma_1^{(\gamma_1, \tilde X_1)}={}_{[a,b]}\Gamma_0^{(\gamma_2, \tilde X_2)}
\end{split}
\end{equation}
Then $\star_{\rm H}$ is defined as 
\begin{equation}\label{horizontal}
{}_{[a,b]}\Gamma^{(\gamma_1, \widetilde X_1)}\star_{\rm H} {}_{[a,b]}\Gamma^{(\gamma_2, \widetilde X_2)}:={}_{[a,b]}\Gamma^{(\gamma_1\circ\gamma_2, \widetilde X_1\circ\widetilde X_2)}.
\end{equation}
Observe that source (respectively target) of ${}_{[a,b]}\Gamma^{(\gamma_1, \tilde X_1)}$ is composable with the source (respectively target) of 
${}_{[a,b]}\Gamma^{(\gamma_2, \tilde X_2)},$  in category ${\mathbf C}^{\rm geod}$ defined in \eqref{1comp}. It is a straightforward  
verification that $\star_{\rm H}$ and $\star$ satify the ``exchange law''
\begin{equation}\label{exchange}
\begin{split}
&\left({}_{[b,c]}\Gamma^{(\gamma_1, \tilde X_1)}\star {}_{[a,b]}\Gamma^{(\gamma_1, \tilde X_1)}\right)\star_{\rm H}\left({}_{[b,c]}\Gamma^{(\gamma_2, \tilde X_2)}\star {}_{[a,b]}\Gamma^{(\gamma_2, \tilde X_2)}\right)\\
=&\left({}_{[b,c]}\Gamma^{(\gamma_1, \tilde X_1)}\star_{\rm H} {}_{[b, c]}\Gamma^{(\gamma_2, \tilde X_2)}\right)\star\left({}_{[a,b]}\Gamma^{(\gamma_1, \tilde X_1)}\star_{\rm H} {}_{[a,b]}\Gamma^{(\gamma_2, \tilde X_2)}\right),
\end{split}
\end{equation}
whenever both sides are well defined. 
\begin{theorem}\label{th:double}
Suppose $M$ be a complete manifold. Let ${\mathbf C}^{\rm geod}$ be the category as described in \eqref{cat1}--\eqref{1comp} and ${\mathbf C}_{(2)}^{\rm geod}$ be the category described in \eqref{2st}--\eqref{cat2} with the partial product $\star_{\rm H}$ defined in \eqref{horizontal}. Then ${\mathbf C}_{(2)}^{\rm geod}$ is a double category over ${\mathbf C}^{\rm geod}$.
\end{theorem}
\section{A physical interpretation of the category  ${\mathbf C}_{(2)}^{\rm geod}$}
In string theory we may consider a string to be an oriented arc on some (Riemannian) manifold $M$, given by $\gamma:[0,1]\rightarrow M.$  String
interactions are described by combining two strings to form a third string (see Figure 1), and the combining process can be either via {\em end-to-end} interaction
or {\em overlap interaction}~\cite{stasheff}. 
\begin{figure}[ht]
\begin{center}
\epsfxsize=5in \epsfysize=1.5in
\rotatebox{0}{\epsfbox{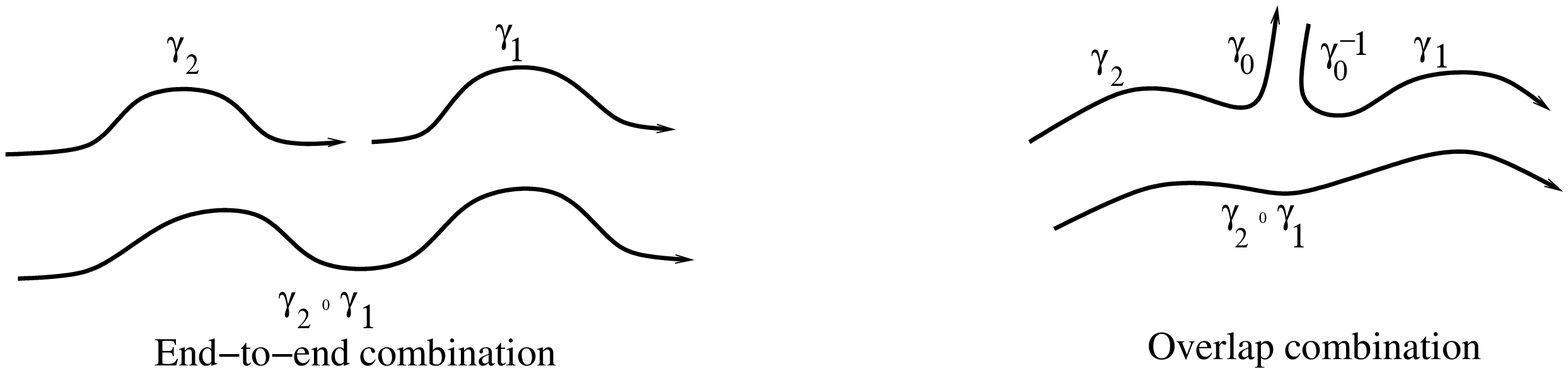}}
\caption{String combinations}
\end{center}
\end{figure}\label{f:inter}

The end-to-end interaction can be described by concatenation of two paths, on the other hand during an overlap
interaction the new string is formed by obliterating the overlapping portions. Now if we recall the back-track erasing method  described in the 
previous section (see \eqref{bacequ}), we immediately see the back-track erasing  essentially describes the overlap interaction. So  
${\mathbb P}^{bt}$ is the category whose objects are points of $M$ and morphisms are strings on $M$. The composition in category ${\mathbb P}^{bt}$
is the interaction of two strings, whereas a morphism in category ${\mathbf C}^{\rm geod}$ is given by a string and  ``velocity"  of the string; the 
element of  $\mathbb R$ present in the morphism of ${\mathbf C}^{\rm geod}$ can be interpreted as an instant of ``time." Thus,  
$(\gamma, \tilde x, a)\in {\rm Mor}({\mathbf C}^{\rm geod})$ can be interpreted as a string $\gamma$ moving with a velocity $\tilde X$ at time $a.$
The category ${\mathbf C}^{\rm geod}$ is slightly more restrictive than category ${\mathbb P}^{bt}$; by \eqref{st1} two morphisms in ${\mathbf C}^{\rm geod}$
are composable when the respective strings are composable in  ${\mathbb P}^{bt}$ and their end points move with the same velocity at a particular instant of time;
 thus by composition in \eqref{1comp} they form a third string moving in a new velocity given by \eqref{1comp}. We will call the two strings {\em interactive} when they are composable in ${\mathbb P}^{bt}$ (i.e. starting-end points coincide) and also they have same velocity at the joining points at any particular time.

Let us now consider the category ${\mathbf C}_{(2)}^{\rm geod}$. A morphism ${}_{[a,b]}\Gamma^{\gamma, \widetilde X}$ in this category is the ``world sheet"
generated between ``time'' $a$ and $b$ by a free moving string $\gamma$ with a velocity $\widetilde X$. The composition $\star$ in \eqref{2comp} implies that the worldsheet generated by the free moving string    $\gamma$ between interval $a$ and $c$ can be decomposed into the worldsheets generated between the intervals $a, b$ and $b, c,$ for some $a\leq b\leq c.$ In other words, we can slice a worldsheet  into worldsheets generated between the intermediate time intervals. On the other hand the partial product or horizontal composition $\star_{\rm H}$ in \eqref{horizontal} has following interpretation. Suppose $\gamma_1,$ $\gamma_2$ are interactive (as defined in the last paragraph) at time $a$. 
Then they also remain interactive at any future instant of time $b$. Moreover, if we consider the third string, say $\gamma_3$, formed by the interaction of $\gamma_1$
, $\gamma_2$ at a time $a$ and the world sheet generated by $\gamma_3$ between time $a, b$, then it is same as ``side ways composition" of two world sheets created by
$\gamma_1$ and $\gamma_2$ between time $a, b.$ Lastly, the exchange law in \eqref{exchange} ensures the necessary consistency between ``slicing" of world sheets and interaction beween strings.

{\bf{ Acknowledgments.} } Author thanks R. Dey and P. Kumar for useful discussions.  Author thanks the anonymous reviewer for his/her comments on the manuscript. Part of the work towards this paper was done while author was in Institut des Hautes \' Etudes Scientifiques (IH\' ES), France.  Author acknowledges a fellowship from the \textit {Jacques Hadamard Mathematical Foundation} during his stay in IH\'ES.

\medskip









\begin{thebibliography}{99}
\bibitem{AbbasWag} H.~Abbaspour, F.~ Wagemann, {\em On 2-Holonomy}, (2012) \url{http://arxiv.org/abs/1202.2292}
\bibitem{Awody} S.~Awody, {\em Category Theory}, Clarendon Press, Oxford (2006).
\bibitem{BaezSchr} J.~Baez, U.~Schreiber,  \textit{ Higher Gauge Theory II: 2-Connections}, \url{http://math.ucr.edu/home/baez/2conn.pdf}
\bibitem{BaezSchr2}J.~Baez, U.~Schreiber,  \textit{Higher Gauge Theory: 2-Connections on 2-Bundles},   \url{http://arxiv.org/abs/hep-th/0412325}
\bibitem{biswas-chat}I. Biswas, S. Chatterjee, \textit{Geometric structures on path spaces},
Int.~J.~Geom.~Meth.~Mod.~Phys {\bf 8} (2011) 1553--1569
\bibitem{brown}R.~Brown, K.~A.~Hardie, K.~H.~Kamps, T.~Porter, \textit{A homotopy double groupoid of a Hausdorff space}, 
Theory and Applications of Categories, {\bf 10}, No. 2 (2002)  71--93
\bibitem{saikat3}S. Chatterjee, A. Lahiri, A. N. Sengupta, \textit{Parallel Transport over Path Spaces},
  Rev. Math. Phys. \textbf{22} (2010) 1033--1059
\bibitem{saikat4}S. Chatterjee, A. Lahiri, A. N. Sengupta, \textit{Path space connections and categorical geometry},
  Jour. Geom. Phys. \textbf{75} (2014) 129--161
\bibitem{Docarmo} M. P. Do Carmo, Translation by F. Fl\"{a}herty, \textit {Riemannian geometry}, Birkh\"{a}user Boston,1992
\bibitem{grove}K.~Grove, \textit {Condition ($C$) for the energy integral on certain path spaces and applications to the theory
of geodesics}, Jour.~Diff.~Geom. {\bf 8} (1973) 207--223
\bibitem{Eliasson} E.~ Halldor \textit{Geometry of manifolds of maps}, Jour.~ Diff.~Geom. {\bf 1} (1967) 169--194
\bibitem{freed}D.~S.~Freed, D. Groisser, \textit{The basic geometry of the manifold of Riemannian metrics and of its quotent by the diffeomorphism group}, Michigan Math. J., {\bf 36}, 3 (1989), 323--344
\bibitem{kob1}S.~ Kobayashi, K.~Nomizu, \textit{Foundations of differential geometry},Vol-1, Interscience Publishers, (1963)
\bibitem{levy} T.~L\'evy,  \textit{Two-Dimensional Markovian Holonomy Fields}, Ast\'erisque {\bf 329} (2010)
\bibitem{mishra}P.~Kumar, \textit{Almost complex structure on path space}, Int.~J.~Geom.~Meth.~Mod.~Phys {\bf 10} (2013) 1--7 
\bibitem{mishra2}P.~Kumar, \textit{Existence of `Darboux chart' on loop space}, \url{http://arxiv.org/abs/1309.2190}

\bibitem{mishra3}P.~Kumar, \textit{Path space as manifold, manifold of maps, infinite dimensional manifold (Geometric analysis on path spaces)}, \url {http://arxiv.org/pdf/1108.2101v1.pdf}, 2011.
\bibitem{Michor} P.~W.~Michor, \textit{Manifolds of differentiable mappings},
Shiva Mathematics Series, 3. Shiva Publishing Ltd., Nantwich, 1980
\bibitem{munoz}V.~M\~unoz, F.~Presas, \textit{Geometric structures on loop and path spaces}, 
Proc. Indian Acad. Sci. (Math. Sci.) {\bf 120} (2010) 417--428
\bibitem{spera}M.~Spera, T.~Wurzbacher,\textit {Good coverings for section spaces of fibre bundles},Topology and its Applications, {\bf 157}, 6 (2010), 1081--1085
\bibitem{stasheff}J.~Stasheff, \textit{An almost groupoid structure for the space of (open) strings and implications for string field theory}, Advances in Homotopy Theory By S. Salamon; B. Steer; W. Sutherland, Cambridge University Press, (1989), 165--172 
\bibitem{wurzbacher}T.~Wurzbacher, \textit{Symplectic geometry of the loop space of a Riemannian manifold}, Jour. Geom. Phys. {\bf 16} (1995), 345--384
 
\end{thebibliography}
\end{document}